\documentclass[11pt,fleqn]{scrartcl}

\usepackage[english]{babel}
\usepackage{amsmath,amsfonts, amsthm,xcolor,amssymb}
\usepackage{paralist}
\usepackage{authblk}

\numberwithin{equation}{section}
\usepackage{mathtools, mathabx}
\usepackage[colorlinks=true,urlcolor=blue, citecolor=blue,linkcolor=blue,linktocpage,pdfpagelabels, bookmarksnumbered,bookmarksopen]{hyperref}
\usepackage{graphicx}
\usepackage{comment}
\usepackage[hyperpageref]{backref}

\usepackage{xcolor}

\newtheorem{thm}{Theorem}[section]
\newtheorem{lem}[thm]{Lemma}

\newtheorem{prop}[thm]{Proposition}

\theoremstyle{definition}
\newtheorem{rem}[thm]{Remark}
\theoremstyle{remark}

\newcommand{\ds}{\displaystyle}
\newcommand{\norm}[1]{\left\Vert#1\right\Vert}

\newcommand{\abs}[1]{\left\vert#1\right\vert}

\newcommand{\R}{\mathbb{R}}

\newcommand{\de}{\partial}
\newcommand{\eps}{\varepsilon}

\def\Xint#1{\mathchoice
    {\XXint\displaystyle\textstyle{#1}}%
    {\XXint\textstyle\scriptstyle{#1}}
    {\XXint\scriptstyle\scriptscriptstyle{#1}}%
    {\XXint\scriptscriptstyle\scriptscriptstyle{#1}}%
      \!\int}
\def\XXint#1#2#3{{\setbox0=\hbox{$#1{#2#3}{\int}$}
    \vcenter{\hbox{$#2#3$}}\kern-.5\wd0}}
\def\dashint{\Xint-}

\DeclareMathOperator{\dive}{div}

\newcommand\restr[2]{{
  \left.\kern-\nulldelimiterspace 
  #1 
  \vphantom{ |} 
  \right|_{#2} 
  }}
{\left\{\begin{array}{@{}l@{}}}{\end{array}\right.}
\makeatletter 
\def\@makefnmark{} 
\makeatother 

\title{Upper and lower bounds for the first Robin eigenvalue of nonlinear elliptic operators}

\author{Rosa Barbato$^*$, Francesco Della Pietra$^*$ \thanks{
\emph{\textrm{Dipartimento di Matematica e Applicazioni ``R. Caccioppoli'', Universit\`a degli studi di Napoli Federico II, Via Cintia, Complesso Universitario Monte S. Angelo, 80126 Napoli, Italy.\\
 E-mail: rosa.barbato2@unina.it, f.dellapietra@unina.it}}}}
\date{}
\begin{document}
\maketitle
\begin{abstract}
 \textbf{Abstract.} Let $\Omega$ be a bounded, smooth domain of $\R^N$, $N\geq 2$. In this paper, we prove some inequalities involving the first Robin eigenvalue of the $p$-laplacian operator. In particular, we prove an upper bound for the first Robin eigenvalue of nonlinear elliptic operators in terms of the first Dirichlet eigenvalue.

 \noindent \textbf{MSC 2020:} 35J25 - 35P15 - 47J10 - 47J30. \\[.2cm]
\textbf{Key words and phrases:}  Nonlinear eigenvalue problems; Robin boundary conditions; Upper and lower bounds.

\end{abstract}

\begin{center}
\begin{minipage}{11cm}
\small
\tableofcontents
\end{minipage}
\end{center}

\section{Introduction}
The main aim of this paper is to prove upper and lower bounds for the first eigenvalue of the $p$-Laplace operator with Robin boundary conditions, namely 
\begin{equation}
\label{variationalaut}
    \lambda_p(\beta,\Omega)=\min_{w\in W^{1,p}(\Omega)\setminus\{0\}}\dfrac{\displaystyle\int_{\Omega}|\nabla w|^p\;dx+\beta\int_{\de\Omega}|w|^p\:d\mathcal{H}^{N-1}}{\displaystyle\int_{\Omega} |w|^p\;dx},
\end{equation}
where $\Omega$ is a bounded, connected, Lipschitz open set of $\R^{N}$, $N\ge2$, $\beta>0$, $1<p<+\infty$. If $u_{\beta}\in W^{1,p}(\Omega)$ is a minimizer of \eqref{variationalaut}, then it satisfies

\begin{equation}
    \label{eigenvalue_problem}
    \begin{cases}
        -\Delta_p u_{\beta}=\lambda_{p}(\beta,\Omega)|u_{\beta}|^{p-2}u_{\beta} & \textrm{in }\Omega\\
        \abs{\nabla u_{\beta}}^{p-2}\dfrac{\de u_{\beta}}{\de \nu}+\beta|u_{\beta}|^{p-2}u_{\beta}=0 &\textrm{on }\de\Omega,
    \end{cases}
\end{equation}
with $\beta$ is a positive parameter and $\nu$ is the outer normal.\\
It is well-known that $u_{\beta}$ has a sign, and $u_{\beta}\in C^{1,\alpha}(\bar\Omega)$ if $\Omega$ is $C^{1,\gamma}$ for some $0<\alpha,\gamma<1$ (see \cite{anle,lieberman}). 

When $\beta=0$, then $\lambda_{p}(\beta,\Omega)=0$ which corresponds to the first trivial Neumann eigenvalue; in the case $\beta\to+\infty$, then $\lambda_{p}(\beta,\Omega)$ tends to $\lambda_p^D(\Omega)$, the first Dirichlet eigenvalue of $-\Delta_{p}$, that is
\begin{equation*}
  \lambda_p^D(\Omega)=\min_{\varphi\in W_0^{1,p}(\Omega)\setminus\{0\}}\dfrac{\displaystyle\int_{\Omega}\abs{\nabla \varphi}^p\;dx}{\displaystyle\int_{\Omega} \abs{\varphi}^p\;dx}.
\end{equation*}
Finding estimates, perhaps the best possible ones, for $\lambda_{p}(\beta,\Omega)$ is a topic of great interest to researchers and several results have been given. Among all, when $\Omega$ is a bounded Lipschitz open set, a Faber-Krahn inequality holds:
\[
\lambda_{p}(\beta,\Omega)\ge \lambda_{p}(\beta,B)
\]
where $B$ is a ball whose Lebesgue measure coincides with that of $\Omega$ (see \cite{bos,bd10,daifu}). Recently, in \cite{della2024sharp} it has been proved that if $\Omega$ is a bounded, mean convex set, then it holds that 
\begin{equation}
\label{hershrobin}
\lambda_p(\beta,\Omega)\ge 
(p-1)\left( \frac{\pi_p}{2}\right)^p\frac{1}{\left(R(\Omega)+\frac{\pi_p}{2}\beta^{-\frac1{p-1}}\right)^p}
,
\end{equation}
where $R(\Omega)$ is the anisotropic inradius of $\Omega$ and $\pi_{p}=\frac{2\pi}{p\sin\frac\pi p}$. Moreover, in \cite{dparkiv} an upper bound of $\lambda_p(\beta,\Omega)$ in terms of the first eigenvalue of a one-dimensional elliptic problem, depending on the perimeter and the volume of a bounded convex set $\Omega$ is given; in particular, if $p=2$ the following estimate holds:
\begin{equation}
    \label{polyarobin}
    \lambda_2(\beta, \Omega) \le \frac{\pi^2}{4} \frac{P^2(\Omega)}{|\Omega|^2} \frac{1}{1+\frac{2 P(\Omega)}{\beta|\Omega|}}
\end{equation}
where $P(\Omega)$ and $|\Omega|$ denotes, respectively, the perimeter and the Lebesgue measure of $\Omega$. We refer the reader also to \cite{liwang,S} for related results. As matter of fact, the inequalities \eqref{hershrobin} and \eqref{polyarobin} generalize the well-known Hersch and P\'olya inequalities for the Dirichlet-Laplace first eigenvalue (see \cite{mana,anona} and the references therein contained).

In this paper, we aim to prove some new upper and lower bounds for $\lambda_p(\beta,\Omega)$. In particular, our first main result is the following:
\begin{thm}
Let $\Omega \subset \R^N$ be an open bounded set, with $C^{1,\gamma}$ boundary, and let $\beta$ be a positive number. Then,
\begin{equation}
\label{teo1intro}
  \dfrac{1}{\lambda_p^{\frac{1}{p-1}}(\beta,\Omega)}\ge \dfrac{1}{\lambda^D_p(\Omega)^{\frac{1}{p-1}}}+\dfrac{\bar{K}}{(\beta P(\Omega))^{\frac{1}{p-1}}},
  \end{equation}
  where $\bar{K}$ is given by
  \begin{equation*}
    \bar{K}=\dfrac{\norm{v}_p^p}{\norm{v}_{p-1}^p},
    \end{equation*}
    and $v$ is a first eigenfunction of $\lambda^{D}_{p}(B)$, defined in a ball $B$ such that $\lambda^{D}_{p}(B)=\lambda^{D}_{p}(\Omega)$.
\end{thm}
In the case $N=p=2$, the inequality \eqref{teo1intro} recovers a result proved in \cite{sperb}. 

In order to prove \eqref{teo1intro}, we use two main tools. The first, is given by a characterization of $\lambda_p(\beta,\Omega)$ by means of a Thompson-type principle (see subsection \ref{thompson}). Then, we use the behavior of the eigenvalue and its eigenfunctions as $\beta\to+\infty$ (see Section \ref{limiti}).

A second result we obtain is again an upper bound for $\lambda_p(\beta,\Omega)$, but in terms of the (Dirichlet) $p$-torsional rigidity of $\Omega$. More precisely, for a given open bounded domain $\Omega$, let $u_{\Omega}\in W^{1,p}_0(\Omega)$ be the unique solution to 
\begin{equation*}
\begin{cases}
    -\Delta_pu_{\Omega}=1 & \textrm{in}\; \Omega\\
    u_{\Omega}=0 & \textrm{on}\; \de \Omega.
\end{cases}
\end{equation*} 
The $p-$torsional rigidity of $\Omega$ is 
\[T_p(\Omega)=\left(\max_{\varphi\in W_0^{1,p}(\Omega)}\dfrac{\left(\ds\int_{\Omega}\abs{\varphi}dx\right)^p}{\displaystyle\int_{\Omega}\abs{\nabla \varphi}^p}dx\right)^{\frac{1}{p-1}}=\int_{\Omega}u_{\Omega}dx=\int_{\Omega}\abs{\nabla u_{\Omega}}^pdx.
\]
Then we obtain the following.
\begin{thm}
Let $\Omega \subset \R^N$ be an open bounded set, with $C^{1,\gamma}$ boundary, and let $\beta$ be a positive number. Then,
\begin{equation}
\label{tors1_intro}
\dfrac{1}{\lambda_p(\beta,\Omega)^{\frac{1}{p-1}}}\geq \dfrac{T_p(\Omega)}{\abs{\Omega}}+\left(\dfrac{|\Omega|}{\beta P(\Omega)}\right)^{\frac{1}{p-1}}.
\end{equation}
\end{thm}
Again, inequality \eqref{tors1_intro} generalizes a result by \cite{sperb} given for $N=p=2$; the main tool to prove \eqref{tors1_intro} still relies on a Thompson-like principle.

As regards a lower bound for the eigenvalue, we get that
if $\Omega \subset \R^N$ is a open bounded set, with $C^{1,\gamma}$ boundary, and $\beta>0$, Then
 \begin{equation}
 \label{lowerintro}
    \dfrac{1}{\lambda_p(\beta,\Omega)}\leq \dfrac{1}{\nu_p}+\dfrac{\abs{\Omega}^{\frac{1}{p-1}}}{\beta^{\frac{1}{p-1}}P(\Omega)^{\frac{1}{p-1}}},
   \end{equation}
   where $\nu_p$ is a suitable value independent of $\beta$ (see Section \ref{sectionlower} for its definition).
This result generalizes the case $p=N=2$ proved in \cite{sperb}. In that case, when $\Omega$ is a suitable symmetric plane domain, $\nu_2$ coincides with the first non-trivial Neumann eigenvalue of the Laplacian.

The proof of \eqref{lowerintro} is based, in addition to the aforementioned Thompson's principle, on the precise behavior of the eigenfunctions of $\lambda_p(\beta,\Omega)$ when $\beta$ goes to $0$. 
Theorem \ref{lambdasubetathm}, Proposition \ref{limit_prob} and Theorem \ref{lambdasubeta} of Section \ref{limiti} state this precise behavior. We show in particular that
\[ 
\lim_{\beta\to 0}  \dfrac{\lambda_p(\beta,\Omega)}{\beta}= \dfrac{P(\Omega)}{\abs{\Omega}},
\]
where, for the sake of completeness, we also compute the limit for negative $\beta$' s (see \cite{smits} for the case $p=2$). 

In summary, the structure of the paper is the following. In Section \ref{preliminari}, we give some preliminary useful results as well as characterization of the eigenvalue in terms of solutions of suitable Robin boundary value problems; furthermore, we prove a convexity principle which will play a key role in the proof of the main results. In Section \ref{limiti}, we prove the limit properties of eigenvalue and eigenfunctions as $\beta\to 0$ and $\beta\to+\infty$. In Sections 4 and 5 we state and prove the quoted upper and lower bounds for $\lambda_p(\beta,\Omega)$.

\section{Preliminary results}
\label{preliminari}
 We first recall the following version of the H\"older inequality proved in \cite[Proposition A.1]{DPOS}.
\begin{prop}
\label{holder}
	Let $1<p, p'<\infty$ be such that $\displaystyle \frac1{p}+\frac1{p'}=1$.
	Assume that $f_1, f_2\>:\>\Omega\to\mathbb R$ and $g_1, g_2\>:\>\de\Omega\to\mathbb R$ are measurable functions satisfying $f_1\in L^p(\Omega)$, $f_2\in L^{p'}(\Omega)$, $g_1\in L^p(\de\Omega, \lambda)$ and $g_2\in L^{p'}(\de\Omega, \lambda)$. Moreover, assume that $\lambda$ is a measurable function on $\partial \Omega$, with $\lambda \ge 0$. Then $f_1f_2\in L^1(\Omega)$, $g_1g_2\in L^1(\de\Omega, \lambda)$ and
	\begin{multline*}
		\int_\Omega|f_1f_2|\, dx+\int_{\de\Omega}\lambda(x)|g_1g_2|\, d\mathcal H^{N-1}\\
		\le \left[\int_\Omega|f_1|^p\, dx+\int_{\de\Omega}\lambda(x)|g_1|^p\, d\mathcal H^{N-1}\right]^{\frac1p}\left[\int_\Omega|f_2|^{p'}\, dx+\int_{\de\Omega}\lambda(x)|g_2|^{p'}\, d\mathcal H^{N-1}\right]^{\frac1{p'}}.
	\end{multline*}
\end{prop}

Second, we recall the following Poincar\'e-Wirtinger inequality (see \cite[Theorem 4.4.6]{ziemer}).
\begin{thm}
    Let $\Omega$ be an open bounded Lipschitz set in $\mathbb R^N$ and $p>1$. Then there exists a positive constant $\Tilde C$ such that for any $u \in W^{1,p}(\Omega)$, with $\int_{\de\Omega}ud\mathcal H^{N-1}=0$ it holds that
    \begin{equation}
        \label{pwirt}
        \int_\Omega |u|^p dx\le \Tilde{C}\int_\Omega |\nabla u|^p dx.
    \end{equation}
\end{thm}

We also recall a well-known trace inequality (see \cite[Theorem $4.2$]{N}): when $\Omega$ is a bounded, Lipschitz domain, there exists a constant $C>0$ such that for all $v \in W^{1,p}(\Omega)$,
\begin{equation*}
	\|v\|_{L^{\frac{(N-1)p}{N-p}}(\partial\Omega)} \le C\|v\|_{W^{1,p}(\Omega)}.
\end{equation*}
Furthermore, the trace embedding is also compact in $L^q(\partial\Omega)$ for $q<\frac{(N-1)p}{N-p}$ (see \cite[Theorem $6.1$]{N}).

Moreover, we are able to obtain a particular trace inequality (see \cite{giorgi_smith} for the case $p=2$) that will be useful in the following. 
 \begin{thm}
    Let $\Omega \subset \R^N$ an open, bounded domain with Lipschitz boundary. Then, there exists $\varepsilon_{\Omega}>0$ such that for every $\varepsilon> \varepsilon_{\Omega}$ there exists $C(\varepsilon)>0$ for which
    \begin{equation}
    \label{trace_inequality}
    \int_{\de \Omega} \abs{u}^p\;d\mathcal{H}^{N-1}\leq \varepsilon\int_{\Omega}\abs{\nabla u}^p\;dx+\dfrac{P(\Omega)}{\abs{\Omega}}(1+C(\varepsilon))\int_{\Omega}\abs{u}^p\;dx.
    \end{equation}
    Moreover, we can choose $C(\eps)\to 0$ as $\eps\to+\infty$.
\end{thm}
\begin{proof}
We proceed by contradiction, hence there exists $\delta>0$ and a sequence $\{u_n\}\subset W^{1,p}(\Omega)$ such that
\[\int_{\de \Omega} \abs{u_n}^pd\mathcal H^{N-1}\ge n\int_{\Omega}\abs{\nabla u_n}^p dx+\dfrac{P(\Omega)}{|\Omega|}(1+\delta)\int_{\Omega}\abs{u_n}^p dx.
\]
If $\int_{\Omega}\abs{\nabla u_n}^p\;dx=0$ for some $n$, then $u_n$ is constant and this is a contradiction, being $\delta>0$. Now, we can assume that $\int_{\Omega}\abs{\nabla u_n}^p\;dx\not=0$ for every $n$. Let us define \[v_n=\dfrac{u_n}{\left(\int_{\Omega}\abs{\nabla u_n}^p\;dx\right)^{\frac{1}{p}}},\]
we have that $\int_{\Omega}\abs{\nabla v_n}^p\;dx=1$ and 
\begin{equation}
\label{tracemaggiore}
\int_{\de \Omega}\abs{v_n}^p\;d\mathcal{H}^{N-1}\geq n+\dfrac{P(\Omega)}{\abs{\Omega}}(1+\delta)\int_{\Omega}\abs{v_n}^p\;dx.
\end{equation}
Now we want to prove that $\left\{\int_{\Omega}\abs{v_n}^pdx\right\}$ is bounded. If this is not true, there exists $\{n_k\}$ such that $\lim_k\int_{\Omega}\abs{v_{n_k}}^pdx = \infty$. Let us define 
\[w_k=\dfrac{v_{n_k}}{\left(\int_{\Omega}\abs{v_{n_k}}^pdx\right)^{\frac{1}{p}}},
\]
we have that 
\[
\lim_{k\to+\infty}\int_{\Omega}\abs{\nabla w_k}^p dx= 0\quad\text{and} \int_{\Omega}\abs{w_k}^p dx=1,
\]
 then $\{w_k\}$ is bounded in $W^{1,p}(\Omega)$ and there exists $\{w_{k_j}\}$ weakly converging in $W^{1,p}(\Omega)$ and strongly in $L^p(\Omega)$ to a function $w\in W^{1,p}(\Omega)$. By uniqueness of the limit $\nabla w=0$ and $w_{k_j}\rightarrow \dfrac{1}{\abs{\Omega}^{\frac{1}{p}}}$ strongly in $W^{1,p}(\Omega)$, hence 
\[
\dfrac{P(\Omega)}{\abs{\Omega}}(1+\delta)\leq\int_{\de \Omega}\abs{w_{k_j}}^p\;d\mathcal{H}^{N-1},
\]
and this is a contradiction, being $\lim_{j\to+\infty}\int_{\de\Omega}\abs{w_{k_j}}^p\;d\mathcal{H}^{N-1}=\dfrac{\abs{\de \Omega}}{\abs{\Omega}}$. So, there exists a subsequence $\{v_{n_k}\}$  such that $\lim_{k\to+\infty}\int_{\Omega}\abs{v_{n_k}}^pdx= c$, for $c\ge 0$. Hence, by using \cite[Theorem 1.5.1.10]{grisvard}, we have that 
\[
\int_{\de \Omega} \abs{v_{n_k}}^p d\mathcal H^{N-1}\le \varepsilon^{1-\frac{1}{p}}\int_{\Omega}\abs{\nabla v_{n_k}}^pdx+K(\eps)\int_{\Omega}\abs{v_{n_k}}^pdx,
\]
for $\eps \in (0,1)$.
Putting together with \eqref{tracemaggiore}, we have that
\[
 n_k + \dfrac{P(\Omega)}{\abs{\Omega}}(1+\delta)\int_{\Omega}\abs{v_{n_k}}^pdx\leq \varepsilon^{1-\frac{1}{p}}+K(\eps)\int_{\Omega}\abs{v_{n_k}}^pdx,
 \]
and this is a contradiction.

\end{proof}

\subsection{A characterization for  \texorpdfstring{$\lambda_p(\beta,\Omega)$}{TEXT}}
\label{thompson}
Let $\Omega$ be a bounded domain of $\R^N$, $N\geq 2$, $\beta>0$, and let $f\in L^{\infty}(\Omega)$ a given function. We consider the following Robin boundary value problem 
\begin{equation}
    \label{plaplaciano}
    \begin{cases}
        -\Delta_p u_f= f & \textrm{in}\;\Omega\\
        \abs{\nabla u_f}^{p-2}\dfrac{\de u_f}{\de \nu}+\beta\abs{u_f}^{p-2}u_f=0 &\textrm{on}\;\de\Omega,
    \end{cases}
\end{equation}
where $1<p$ and $\beta$ is a positive parameter. By standard existence results, there exists a unique function $u_f\in W^{1,p}(\Omega)$ which solves \eqref{plaplaciano}, in the sense that
\begin{equation}
    \label{weak_formulation}
    \int_{\Omega} \abs{\nabla u_f}^{p-2}\nabla u_f \nabla \varphi \, dx + \beta \int_{\de \Omega} \abs{u_f}^{p-2} u_f \varphi \, d\mathcal{H}^{N-1}(x) = \int_{\Omega} f \varphi \, dx \quad \forall \varphi \in W^{1,p}(\Omega).
\end{equation}
Moreover, $u_f\ge 0$ if $f\ge 0$, and by classical regularity results, if $\Omega$ has $C^{1,\gamma}$ boundary, then $u_{f}\in C^{1,\alpha}(\overline\Omega)$ (\cite{lieberman}).


In the limiting cases $\beta=0$ and $\beta=\infty$, we recover the Neumann and the Dirichlet boundary conditions respectively.

Now we are in a position to state and prove a Thompson principle for the Robin $p-$Laplacian. For a given $f\in L^{\infty}(\Omega)$, let $u_f\in W^{1,p}(\Omega)$ be the solution to \eqref{plaplaciano}, and define the quantity 
\begin{equation*}
J_f(\beta):=\int_{\Omega}\abs{\nabla u_f}^p\;dx+\beta\int_{\de \Omega} \abs{u_f}^p\;d\mathcal{H}^{N-1}=\int_\Omega f u_f dx
\end{equation*} 
where the last equality follows from the equation, using $u_f$ as test function.


Given $f\in L^{\infty}(\Omega)$, let us consider the set $\mathcal V_{f}$ of all the continuous vector functions $V\in C(\overline\Omega)$ such that $-\dive V=f$, in the sense that
\begin{equation}
\label{weakV}
\int_\Omega V\cdot \nabla \varphi\, d x+\int_{\de\Omega} \varphi V \cdot \nu \,d\mathcal H^{N-1}=\int_{\Omega}f\varphi\,dx,\quad \forall \varphi \in C^{1}(\Omega).
\end{equation}

\begin{lem}
Let $\Omega$ be a bounded, $C^{1,\gamma}$ domain of $\R^N$, $N\geq 2$, and let $\beta$ be a positive parameter. If $f\in L^{\infty}(\Omega)$ is a nonnegative function, we have that
\begin{equation}
\label{thomson}
    J_f(\beta)=\min_{V\in \mathcal V_{f}}\left\{\int_{\Omega}\abs{V}^{p'}\;dx+\beta^{-\frac{1}{p-1}}\int_{\de \Omega}\abs{V\cdot \nu}^{p'}\;d\mathcal{H}^{N-1}\right\}.
\end{equation}
\end{lem}
\begin{proof}
Let $u_{f}$ be the solution to the problem \eqref{plaplaciano}. Then by definition of $J_{f}$ and \eqref{weakV} with $\varphi=u_{f}$ it holds that
\begin{equation*}
J_{f}(\beta)=\int_{\Omega}fu_{f}dx=\int_{\Omega} V\cdot \nabla u_{f}\,dx+\int_{\de \Omega}u_f{V\cdot \nu}\;d\mathcal{H}^{N-1},
\end{equation*}
for any $V\in \mathcal V_{f}$. Hence, applying Proposition \ref{holder} with 
$f_1=\abs{\nabla u_f}$, $f_2=\abs{V}$, $g_1=\beta^{\frac 1p}u_f$ and $g_2=\beta^{-\frac1p}\abs{V \cdot \nu}$, we have
\begin{multline*}
J_f(\beta)\leq \left(\int_{\Omega}\abs{V}^{\frac{p}{p-1}}\;dx+\beta^{-\frac{1}{p-1}}\int_{\de\Omega}\abs{V\cdot \nu}^{\frac{p}{p-1}}\;d\mathcal{H}^{N-1}\right)^{\frac{p-1}{p}} \times \\ \times
\left(\int_{\Omega}\abs{\nabla u_f}^p\;dx+\beta\int_{\de \Omega}\abs{u_f}^p\;d\mathcal{H}^{N-1}\right)^{\frac{1}{p}},
\end{multline*}
hence
\[
J_f(\beta)\leq \int_{\Omega}\abs{V}^{\frac{p}{p-1}}\;dx+\beta^{-\frac{1}{p-1}}\int_{\de\Omega}\abs{V\cdot \nu}^{\frac{p}{p-1}}\;d\mathcal{H}^{N-1}.
\]

Finally, if $V=\abs{\nabla u_f}^{p-2}\nabla u_f$, then $V\in \mathcal V_{f}$ and the equality in the above inequality holds, hence the thesis follows.

\end{proof}

Now, we prove that the first Robin eigenvalue of the $p$-Laplace operator can be written in terms of a maximum problem involving the functional $J_f(\beta)$.
\begin{prop}
It holds that
\begin{equation}
\label{autovalore}
\dfrac{1}{\lambda_p^{\frac{1}{p-1}}(\beta)}
= 
\max \left\{
\dfrac{J_f(\beta)}{\displaystyle\int_{\Omega}f^{p'}\;dx},\; f\in L^{\infty}(\Omega),\,f\ge 0, f\not\equiv 0 \right\}.
\end{equation}
\end{prop}
%
\begin{proof}
Let $f\ge 0$, $f\in L^{\infty}(\Omega)$, $f\not\equiv 0$ and $u_f$ be a solution to \eqref{plaplaciano}. 
Then by the H\"older inequality it holds that 
\begin{equation}
\label{pass1}
\begin{aligned}
    J_f(\beta)=\int_{\Omega}u_f f\;dx&\le 
    \left(\int_{\Omega}u_f^p\;dx\right)^{\frac{1}{p}}\left(\int_{\Omega}f^{p'}\;dx\right)^{\frac{1}{p'}}
    \end{aligned}
\end{equation}
On the other hand, using $u_{f}$ as test function in the variational formulation \eqref{variationalaut} of $\lambda_{p}(\beta,\Omega)$ it holds that
\begin{equation}
\label{pass2}
\int_{\Omega} u_{f}^{p} dx \le \dfrac{\ds J_{f}(\beta,\Omega)}{\ds \lambda_{p}(\beta)}
\end{equation}
Using \eqref{pass2} in \eqref{pass1} and rearranging the terms we get
\begin{equation}
\label{pass3}
\dfrac{J_f(\beta)}{\displaystyle\int_{\Omega}f^{p'}\;dx} \le \dfrac{1}{\lambda_p^{\frac{1}{p-1}}(\beta,\Omega)}.
\end{equation} 
By observing that if $f=\lambda_{p}(\beta)u_\beta^{p-1}$, where $u_\beta$ is a first positive eigenfunction of \eqref{eigenvalue_problem}, the equality in \eqref{pass3} holds and the thesis follows.
\end{proof}
A key tool related to $J_f(\beta)$ is given by the following convexity property:
\begin{lem}
Let $\Omega$ be a bounded open set in $\mathbb R^N$, with $C^{1,\gamma}$ boundary. Then for any $f\in L^{\infty}(\Omega)$, $f\ge 0$, the function $J_f$ satisfies
\begin{equation}
\label{convexity}
J_f(\beta)\leq J_f(\alpha)+\left(\dfrac{1}{\beta^{\frac{1}{p-1}}}-\dfrac{1}{\alpha^{\frac{1}{p-1}}}\right)H(\alpha),
\end{equation}
for any $\alpha,\beta>0$ and for some function $H(\alpha)\ge 0$.
\end{lem}
\begin{proof}
To prove \eqref{convexity}, let $f$ be fixed and let us denote with $u_{f,\alpha}\in C^{1,\gamma}(\bar\Omega)$ be the solution to \eqref{plaplaciano} where the coefficient in boundary condition is given by $\alpha$.  By the Thompson principle \eqref{thomson}, if we choose $V=\abs{\nabla u_{f,\alpha}}^{p-2}\nabla u_{f,\alpha}$, we have that
\begin{align*}
    J_f(\beta)&\leq \int_{\Omega}\abs{\nabla u_{f,\alpha}}^{p}\;dx+\beta^{-\frac{1}{p-1}}\int_{\de \Omega}\abs{\abs{\nabla u_{f,\alpha}}^{p-2}\nabla u_{f,\alpha}\cdot \nu}^{\frac{p}{p-1}}\;d\mathcal{H}^{N-1}\\
    &= \int_{\Omega}\abs{\nabla u_{f,\alpha}}^{p}+\alpha^{-\frac{1}{p-1}}\int_{\de \Omega}\abs{\abs{\nabla u_{f,\alpha}}^{p-2}\nabla u_{f,\alpha}\cdot \nu}^{\frac{p}{p-1}}d\mathcal{H}^{N-1}\\& +\left(\dfrac{1}{\beta^{\frac{1}{p-1}}}-\dfrac{1}{\alpha^{\frac{1}{p-1}}}\right)\int_{\de \Omega}\abs{\abs{\nabla u_{f,\alpha}}^{p-1}\nabla u_{f,\alpha}\cdot\nu}^{\frac{p}{p-1}}d\mathcal{H}^{N-1}\\
    &\leq J_f(\alpha)+\left(\dfrac{1}{\beta^{\frac{1}{p-1}}}-\dfrac{1}{\alpha^{\frac{1}{p-1}}}\right)\int_{\de \Omega}\abs{\abs{\nabla u_{f,\alpha}}^{p-2}\nabla u_{f,\alpha}\cdot\nu}^{\frac{p}{p-1}}d\mathcal{H}^{N-1}.
\end{align*}

Denoting by $H(\alpha)$ the function $\int_{\de \Omega}\abs{\abs{\nabla u_{f,\alpha}}^{p-2}\nabla u_{f,\alpha}\cdot \nu}^{\frac{p}{p-1}}d\mathcal{H}^{N-1}$, the thesis follows.
\end{proof}

\section{Limit properties when
\texorpdfstring{$\beta\to 0$}{TEXT}
or
\texorpdfstring{$\beta\to +\infty$}{TEXT}
}
\label{limiti}
The following proposition is well-known. For the sake of completeness, we give also a proof.
\begin{prop}
    Let $\Omega$ be a bounded, Lipschitz open set in $\R^N$. Then
    \begin{equation}
    \label{limbeta}
            \lim_{\beta\to+\infty} \lambda_p(\beta,\Omega)=\lambda_p^D(\Omega).   
    \end{equation}
\end{prop}
\begin{proof}
    If $u_\beta$ is a first positive eigenfunction for $\lambda_p(\beta,\Omega)$ such that $\|u_\beta\|_{L^p(\Omega)}=1$, it holds that
    \[
    \lambda_p(\beta,\Omega)= \int_\Omega|\nabla u_\beta|^pdx+\beta\int_\Omega u_\beta^pd\mathcal H^{N-1} \le \lambda^D(\Omega),\quad \forall \beta\in\R.
    \]
    Then $u_\beta$ is bounded in $W^{1,p}(\Omega)$; it weakly converges in $W^{1,p}(\Omega)$ and strongly in $L^p(\Omega)$ and in $L^p(\de\Omega)$ to a function $u_\infty\in W_0^{1,p}(\Omega)$, and $\|u_\infty\|_{L^p(\Omega)}=1$. By  the definition of $\lambda^D(\Omega)$ and semicontinuity, it holds that
    \begin{multline*}
        \lambda^D(\Omega)\le \int_\Omega |\nabla u_\infty|^pdx \le \liminf_{\beta\to +\infty} \left[\int_\Omega|\nabla u_\beta|^pdx+\beta\int_\Omega u_\beta^pd\mathcal H^{N-1} \right]\le \\ \le \limsup_{\beta\to +\infty} \left[ \int_\Omega|\nabla u_\beta|^pdx+\beta\int_\Omega u_\beta^pd\mathcal H^{N-1} \right]\le \lambda^D(\Omega),
    \end{multline*}
    and the proof is complete.
\end{proof}
\begin{prop}
\label{proplimbeta}
Let $\Omega$ be a bounded, Lipschitz open set in $\R^N$, and let $u_{f,\beta}\in W^{1,p}(\Omega)$ be a solution to the problem \eqref{plaplaciano}, with $f\in L^{p'}(\Omega)$. Then there exists a function $u_{f,\infty}\in W^{1,p}_0(\Omega)$ that satisfies
\begin{equation}
\label{dirichletf}
\begin{cases}
-\Delta_{p} u_{f,\infty}= f&\text{in }\Omega,\\
u_{f,\infty}=0&\text{on }\de\Omega
\end{cases}
\end{equation}
and it holds that
\begin{equation}
    \label{dirichletf2}
\lim_{\beta\to +\infty} J_f(\beta)= 
\lim_{\beta\to +\infty} \int_\Omega |\nabla u_{f,\beta}|^pdx= \int_\Omega |\nabla u_{f,\infty}|^pdx.
\end{equation} 
\end{prop}
\begin{proof}
    Using $u_{f,\beta}$ as test function in \eqref{weak_formulation}, by definition of $\lambda_p(\beta,\Omega)$ and the H\"older inequality, we get
    \begin{multline}
    \label{passlimite}
     \lambda_p(\beta,\Omega)\int_\Omega|u_{f,\beta}|^p\,dx\le   \int_\Omega |\nabla u_{f,\beta}|^p \,dx+\beta\int_{\de\Omega} |u_{f,\beta}|^{p}\,d\mathcal H^{N-1} = \int_\Omega fu_{f,\beta}\,dx \le \\ \le \|f\|_{L^{p'}(\Omega)}\|u_{f,\beta}\|_{L^p(\Omega)}.
    \end{multline}
        Then, recalling \eqref{limbeta}, \eqref{passlimite} implies
    \[
    \int_\Omega|u_{f,\beta}|^pdx\le C \quad\text{as }\beta\to +\infty,
    \]
 and using again \eqref{passlimite} it holds that $u_{f,\beta}$ is bounded in $W^{1,p}(\Omega)$ as $\beta\to +\infty$. Hence $u_{f,\beta}$ weakly converges in $W^{1,p}(\Omega)$ and strongly converges in $L^p(\Omega)$ and $L^p(\de\Omega)$ to a function $u_{f,\infty}\in W^{1,p}(\Omega)$. Moreover, by \eqref{passlimite} again it holds that $u_{f,\infty}\in W^{1,p}_0(\Omega)$. Finally, we can pass to the limit in \eqref{weak_formulation} in order to get that $u_{f,\infty}$ solves \eqref{dirichletf},
 and 
 \begin{multline*}
   \lim_{\beta\to+\infty} J_f(\beta)= \lim_{\beta\to +\infty} \left[\int_\Omega |\nabla u_{f,\beta}|^pdx+\beta \int_{\de\Omega}|u_{f,\beta}|^p d\mathcal H^{N-1}\right] \\ =
  \lim_{\beta\to+\infty} \int_\Omega f u_{f,\beta}dx= \int_\Omega f u_{f,\infty}dx=    \int_\Omega |\nabla u_{f,\infty}|^pdx.
 \end{multline*}
Hence \eqref{dirichletf2} is in force.
\end{proof}
Now, we show the behavior of $\lambda_p(\beta,\Omega)$ as $\beta\to0^+$.
\begin{thm} 
\label{lambdasubetathm}
Let $\Omega$ be a bounded Lipschitz domain in $\R^{N}$ and $\lambda_p(\beta,\Omega)$ be the first Robin eigenvalue, with $\beta>0$. Then
    \begin{equation}
    \label{lambdasubeta}
      \lim_{\beta\to 0^{+}}  \dfrac{\lambda_p(\beta,\Omega)}{\beta}= \dfrac{P(\Omega)}{\abs{\Omega}}.
    \end{equation}
\end{thm}
\begin{proof}
    Let $u_\beta$  be a first positive eigenfunction of $\lambda_{p}(\beta,\Omega)$. By using as test function $\varphi\equiv 1$ in the equation solved by $u_{\beta}$, we have that
    \[\lambda_p(\beta,\Omega)\int_{\Omega}u_{\beta}^{p-1}dx=\beta\int_{\de \Omega}u_{\beta}^{p-1}dx.\]
    For $\beta\rightarrow 0^+$, $\lambda_p(\beta,\Omega)\rightarrow 0.$ We claim that $u_{\beta}\rightarrow c$, where $c$ is a constant, for $\beta \rightarrow 0^+$. 
    To prove the claim, let us fix $||u_\beta||_{L^p(\Omega)}=1$. Hence
    \[
    \int_{\Omega}\abs{\nabla u_{\beta}}^p dx\le \int_{\Omega}\abs{\nabla u_{\beta}}^pdx +\beta\int_{\de\Omega}{u_{\beta}}^p d\mathcal H^{N-1}=\lambda_p(\beta,\Omega),
    \]
 then $||\nabla u_\beta||_{L^p(\Omega)}$ is bounded, and $u_\beta$ is bounded in $W^{1,p}(\Omega)$. By compactness, there exists a subsequence, still denoted by $u_\beta$, and a function $u_{0}\in W^{1,p}(\Omega)$ such that
\[
 \begin{array}{ll}
 \nabla u_{\beta}\rightharpoonup \nabla u_{0} & \text{weakly in }L^{p}(\Omega),\\
  u_{\beta}\rightarrow u_{0} & \text{strongly in }L^{q}(\Omega),\quad 1\le q\le p\\
    u_{\beta}\rightarrow u_{0} & \text{strongly in }L^{p}(\de\Omega),\quad 1\le q\le p.
 \end{array}
 \] 
  Hence
    \begin{multline*}
        \int_{\Omega}\abs{\nabla u_{0}}^pdx\le \liminf_{\beta\to 0^{+}}\left(\int_{\Omega}\abs{\nabla u_{\beta}}^pdx+\int_{\de \Omega}\beta \abs{u_{\beta}}^pd\mathcal H^{N-1}\right)\\=\liminf_{\beta \to 0^{+}}\lambda_p(\beta,\Omega)=0.    
    \end{multline*}    
    Then, $||\nabla u_{0}||_{L^p(\Omega)}= 0$ and $u_{0}=c$ is constant in $\overline\Omega$. Using in particular that
    \[
    \dfrac{\lambda_p(\beta,\Omega)}{\beta}=\dfrac{\ds\int_{\de \Omega}u_{\beta}^{p-1}d\mathcal H^{N-1}}{\ds\int_{\Omega}u_{\beta}^{p-1}dx},
    \]
    passing to the limit we have that
    \[
\lim_{\beta\to 0^+}    \dfrac{\lambda_p(\beta,\Omega)}{\beta}= \dfrac{P(\Omega)}{\abs{\Omega}},
    \]
    that is the thesis.
\end{proof}

Now, we show the behavior of the solutions of \eqref{plaplaciano} as $\beta\to 0^{+}$.
\begin{prop}
\label{limit_prob}
Let $\Omega$ be a bounded, Lipschitz domain, $0\le f\in L^{\infty}(\Omega)$ and $\beta>0$, and consider $u_{f,\beta}\in W^{1,p}(\Omega)$ a solution to \eqref{plaplaciano}. Then the sequence $\psi_\beta = u_{f,\beta}-m_\beta$, with $m_\beta=\frac{1}{P(\Omega)}\int_{\de\Omega} u_{f,\beta}d\mathcal H^{N-1}$, weakly converges in $W^{1,p}(\Omega)$, as $\beta$ goes to $0$, to a function $v\in W^{1,p}(\Omega)$ which satisfies
\begin{equation}
\label{limitebeta0}
\begin{cases}
    -\Delta_p v=f &\textrm{in}\;\Omega\\[.2cm]
    \abs{\nabla v}^{p-2}\dfrac{\de v}{\de \nu}=-\dfrac{1}{P(\Omega)}\displaystyle\int_{\Omega}f\:dx &\textrm{on}\;\de \Omega,\\[.2cm]
 \ds\int_{\de\Omega} v d\mathcal H^{N-1}=0.
\end{cases}
\end{equation}
Moreover, 
\[ 
\lim_{\beta \rightarrow 0^+} \int_{\Omega}\abs{\nabla u_{f,\beta}}^p\;dx=\int_{\Omega}\abs{\nabla v}^pdx.
\]
\end{prop}
\begin{proof}
Being $f\ge 0$, $\beta>0$, then $u_{f,\beta}\ge 0$. By the weak formulation \eqref{weak_formulation} 
we have that 
\[
\int_{\Omega}\abs{\nabla u_{f,\beta}}^pdx+\beta \int_{\de \Omega}{u^p_{f,\beta}}d\mathcal H^{N-1}=\int_{\Omega}fu_{f,\beta}dx.
\]
By using the definition of the first Robin eigenvalue of the $p$-Laplace operator, we can write  
\[
\int_{\Omega}\abs{\nabla u_{f,\beta}}^pdx+\beta \int_{\de \Omega}{u^p_{f,\beta}} d\mathcal H^{N-1}\le \left(\dfrac{1}{\lambda_p(\beta,\Omega)}\right)^{\frac{1}{p-1}}\int_{\Omega}f^{p'}dx.
\]
Hence recalling \eqref{lambdasubeta},
\[
\frac{1}{\beta}\int_{\Omega}\abs{\nabla[\beta^\frac{1}{p-1}u^p_{f,\beta}]} dx + \int_{\de \Omega}[\beta^\frac{1}{p-1}u_{f,\beta}]^pd\mathcal H^{N-1}\le C,
\]
for some positive constant $C$ independent of $\beta$. Hence,
 for $\beta \rightarrow 0^+$, we have that $\nabla[\beta^\frac{1}{p-1}u_{f,\beta}]\rightharpoonup 0$ weakly in $L^p(\Omega)$, and $\beta^\frac{1}{p-1}u_{f,\beta}$ converges in $L^p(\de \Omega)$ to a constant $K$. 
Hence from the equation, using $\psi\equiv 1$ as test function, 
\[
\int_{\Omega}f=\int_{\de \Omega} [\beta^\frac{1}{p-1}u_{f,\beta}]^{p-1}d\mathcal H^{N-1},
\]
then for $\beta\rightarrow 0^+$, we obtain $K=\left(\dfrac{1}{P(\Omega)}\displaystyle\int_{\Omega}fdx\right)^{\frac{1}{p-1}}$.

Now, let 
\[
\psi_\beta=u_{f,\beta} - m_\beta, 
\]
where $m_\beta=\frac{1}{P(\Omega)}\int_{\de\Omega} u_{f,\beta} d\mathcal H^{N-1}$,
be a test function for \eqref{weak_formulation}. Then
it holds that
\begin{equation}
\label{testpsibeta0}
\int_\Omega |\nabla \psi_{\beta}|^p dx+ \beta \int_{\de\Omega} u_{f,\beta}^{p-1}\psi_\beta d\mathcal H^{N-1} = \int_\Omega f\psi_\beta dx.    
\end{equation}
Obviously, $\psi_\beta$ has vanishing mean value on $\de\Omega$; this implies that
\begin{equation}
\label{pass22}
    \int_{\de\Omega}  u_{f,\beta}^{p-1}\psi_\beta d\mathcal H^{N-1} = 
\int_{\de\Omega}  (u_{f,\beta}^{p-1}-m_\beta^{p-1})\psi_\beta  d\mathcal H^{N-1} \ge 0 
\end{equation}
Hence, from \eqref{testpsibeta0}, \eqref{pass22}, the H\"older inequality and the Poincar\'e-Wirtinger inequality \eqref{pwirt}, we have
\begin{equation*}
\norm{\nabla \psi_{\beta}}^{p}_{L^p(\Omega)}\le \norm{f}_{L^{p'}(\Omega)}\norm{\psi_{\beta}}_{L^p(\Omega)}\le C\norm{f}_{L^{p'}(\Omega)}\norm{\nabla\psi_{\beta}}_{L^p(\Omega)},
\end{equation*}
and then, using also again \eqref{pwirt},
\[
\norm{\psi_{\beta}}_{L^p(\Omega)}\le C\norm{\nabla \psi_{\beta}}_{L^p(\Omega)}\le C \|f\|_{L^{p'}(\Omega)}
\]

This implies that, up to a subsequence, $\psi_\beta$ converges weakly in $W^{1,p}(\Omega)$ and strongly in $L^p(\Omega)$ and in $L^p(\de\Omega)$, to a function $v\in W^{1,p}(\Omega)$. 
By using also the convergence of $\beta^{\frac{1}{p-1}}u_{f,\beta}$ to $K$ in $L^p(\de\Omega)$, passing to the limit in 
\[
\int_\Omega |\nabla \psi_\beta|^{p-2}\nabla \psi_\beta \cdot \nabla \varphi dx +\beta\int_{\de\Omega}u_{f,\beta}^{p-1}\varphi d\mathcal H^{N-1}=\int_\Omega f\varphi dx
\]
we get
\[
\int_\Omega |\nabla v|^{p-2}\nabla v \cdot \nabla \varphi +K^{p-1}\int_{\de\Omega}\varphi \, d\mathcal H^{N-1}=\int_\Omega f\varphi dx,
\]
that means that the function $v$ is a solution to \eqref{limitebeta0}. Again, passing to the limit in \eqref{testpsibeta0}, we have
\[
\lim_{\beta \to 0^+}
\int_{\Omega}\abs{\nabla u_{f,\beta}}^p\;dx=\int_{\Omega}fv-K^{p-1}\int_{\de \Omega}v=\int_{\Omega}\abs{\nabla v}^p\;dx
\]
and the proof is concluded.
\end{proof}

\begin{rem}[The case $\beta<0$]
For the sake of completeness, we consider the behavior of $\lambda_p(\beta,\Omega)$ also when $\beta\to0^-$. In this case, the trace inequality \eqref{trace_inequality} allows us to prove that the limit \eqref{lambdasubeta} holds true also when $\beta\to 0^{-}$. This is proved in the following result.
\end{rem}
 \begin{thm}
Let $\Omega$ be a bounded Lipschitz domain in $\R^{N}$ and $\lambda_p(\beta,\Omega)$ be the first Robin eigenvalue, with $\beta<0$. Then
    \begin{equation*}
      \lim_{\beta\to 0^{-}}  \dfrac{\lambda_p(\beta,\Omega)}{\beta}= \dfrac{P(\Omega)}{\abs{\Omega}}.
    \end{equation*}
   \end{thm}
\begin{proof}
    Let $u_\beta$ be a positive eigenfunction of $\lambda_p(\beta,\Omega)$ with $\norm{u_{\beta}}_{L^p(\Omega)}=1$. Then
    \[
    \lambda_p(\beta,\Omega)=\int_{\Omega}\abs{\nabla u_{\beta}}^pdx+\beta\int_{\de \Omega}\abs{u_{\beta}}^pd\mathcal H^{N-1}.
    \]
Using the trace inequality \eqref{trace_inequality} with $\eps=-\frac{1}{\beta}$ (as $|\beta|$ is small)
we get
\[
\frac{P(\Omega)}{|\Omega|} \le \frac{\lambda_p(\beta,\Omega)}{\beta}
\le \frac{P(\Omega)}{|\Omega|}\left(1+C\left(\frac{1}{-\beta}\right)\right).
\]
The thesis follows, being $C(-1/\beta)\to 0$ as $\beta\to 0^-$.
\end{proof}

\section{Upper bounds for \texorpdfstring{$\lambda_p(\beta,\Omega)$}{TEXT}}
In this Section, we want to get upper bounds for $\lambda_p(\beta,\Omega)$. 
The first general result is the following.
\begin{thm}
Let $\Omega \subset \R^N$ be an open bounded set, with $C^{1,\gamma}$ boundary, $\beta$ be a positive number, and $f\in L^{\infty}(\Omega)$, with $f\ge0$. then
\begin{equation}
\label{stimabase}
\dfrac{1}{\lambda_p(\beta,\Omega)^{\frac{1}{p-1}}}\geq \dfrac{\displaystyle\int_{\Omega}\abs{\nabla u_{f,\infty}}^p\;dx}{\displaystyle\int_{\Omega}f^{\frac{p}{p-1}}\;dx}+\dfrac{\beta^{-\frac{1}{p-1}}}{P(\Omega)^{\frac{1}{p-1}}}\dfrac{\displaystyle\left(\int_{\Omega}f\:dx\right)^{\frac{p}{p-1}}}{\displaystyle\int_{\Omega}f^{\frac{p}{p-1}}\:dx},
\end{equation}
where $P(\Omega)=\mathcal H^{N-1}(\de\Omega)$ denotes the perimeter of $\Omega$, and $u_{f,\infty}$ satisfies \eqref{dirichletf}.
\end{thm}
\begin{proof}
Let $f$ be fixed. Then by using the convexity of $J_f(\beta)$,
\[
J_f(\alpha)\leq J_f(\beta)+\left(\dfrac{1}{\alpha^{\frac{1}{p-1}}}-\dfrac{1}{\beta^{\frac{1}{p-1}}}\right)\int_{\de \Omega}\abs{\abs{\nabla u_{f,\beta}}^{p-2}\left(\dfrac{\de u_{f,\beta}}{\de \nu}\right)}^{\frac{p}{p-1}}d\mathcal H^{N-1}.
\]
If we multiply for $\beta^{\frac{1}{p-1}}$,
\[\beta^{\frac{1}{p-1}}J_f(\alpha)\leq \beta^{\frac{1}{p-1}}J_f(\beta)+\left(-1+\left(\dfrac{\beta}{\alpha}\right)^{\frac{1}{p-1}}\right)\int_{\de\Omega}\abs{\abs{\nabla u_{f,\beta}}^{p-2}\left(\dfrac{\de u_{f,\beta}}{\de \nu}\right)}^{\frac{p}{p-1}}d\mathcal H^{N-1}, 
\]
then for $\alpha$ that goes to infinity
\[
\beta^{\frac{1}{p-1}}J_f(\beta)\geq \beta^{\frac{1}{p-1}}J_f(\infty)+\int_{\de \Omega}\abs{\abs{\nabla u_{f,\beta}}^{p-2}\left(\dfrac{\de u_{f,\beta}}{\de \nu}\right)}^{\frac{p}{p-1}}d\mathcal H^{N-1}. 
\]
We know by Proposition \ref{proplimbeta} that $J_f(\infty)=\ds\int_{\Omega}\abs{\nabla u_{f,\infty}}^pdx$ and using the H\"older inequality and the equation solved by $u_{f,\beta}$
\begin{multline*}
\int_{\de \Omega}\abs{\abs{\nabla u_{f,\beta}}^{p-2}\left(\dfrac{\de u_{f,\beta}}{\de \nu}\right)}^{\frac{p}{p-1}}d\mathcal H^{N-1}\\ \ge \dfrac{1}{P(\Omega)^{\frac{1}{p-1}}}\left(\int_{\de \Omega}\abs{\abs{\nabla u_{f,\beta}}^{p-2}\left(\dfrac{\de u_{f,\beta}}{\de \nu}\right)}d\mathcal{H}^{N-1}\right)^{\frac{p}{p-1}}\; = \dfrac{1}{P(\Omega)^{\frac{1}{p-1}}}\left(\int_{\Omega}f\:dx\right)^{\frac{p}{p-1}}.
\end{multline*}
Then, putting all together, we obtain
\[
J_f(\beta)\geq \int_{\Omega}\abs{\nabla u_{f,\infty}}^p dx+\dfrac{\beta^{-\frac{1}{p-1}}}{P(\Omega)^{\frac{1}{p-1}}}\left(\int_{\Omega}f\,dx\right)^{\frac{p}{p-1}}.
\]
So finally, by using \eqref{autovalore} we get that
\[
\dfrac{1}{\lambda_p(\beta,\Omega)^{\frac{1}{p-1}}}\geq \dfrac{\displaystyle\int_{\Omega}\abs{\nabla u_{f,\infty}}^p\;dx}{\displaystyle\int_{\Omega}f^{\frac{p}{p-1}}\;dx}+\dfrac{\beta^{-\frac{1}{p-1}}}{P(\Omega)^{\frac{1}{p-1}}}\dfrac{\displaystyle\left(\int_{\Omega}f\:dx\right)^{\frac{p}{p-1}}}{\displaystyle\int_{\Omega}f^{\frac{p}{p-1}}\:dx}
\]
that is the result.
\end{proof}
Now we are in position to prove the main bounds for $\lambda_{p}(\beta,\Omega)$.
So, let $u_\infty$ be a first positive Dirichlet eigenfunction of $-\Delta_{p}$, that is a positive function such that
\begin{equation}
\label{dirichlet}
    \begin{cases}
        -\Delta_p u_\infty=\lambda_p^D(\Omega) u_\infty^{p-1} & \textrm{in } \Omega\\
        u=0 & \textrm{on } \de \Omega.
    \end{cases}
\end{equation}
We recall that the following reverse H\"older inequality holds:
\begin{thm}[\cite{reverseholder}]
    \label{reverseholder}
    For $0<q<r\leq\infty$ we have
    \[
    {\norm{u_\infty}}_r \leq \Tilde{K} {\norm {u_\infty}}_q,
    \]
    with 
    \begin{equation}
    \label{costantek}
    \Tilde{K}=\Tilde{K}(N, p, q, r, \lambda^D_p(\Omega))=\dfrac{\norm{v}_r}{\norm{v}_q},
    \end{equation}
    and $v$ is a first eigenfunction of \eqref{dirichlet} in a ball $B$ such that $\lambda^{D}_{p}(B)=\lambda^{D}_{p}(\Omega)$.
\end{thm}
We can now prove the following main results.

\begin{thm}
Let $\Omega \subset \R^N$ be an open bounded set, with $C^{1,\gamma}$ boundary, and let $\beta$ be a positive number. Then,
\begin{equation}
\label{upper1}
  \dfrac{1}{\lambda_p^{\frac{1}{p-1}}(\beta,\Omega)}\geq \dfrac{1}{\lambda^D_p(\Omega)^{\frac{1}{p-1}}}+\dfrac{\bar{K}}{(\beta P(\Omega))^{\frac{1}{p-1}}},
  \end{equation}
  where $\bar K=\Tilde{K}^p$, and $\Tilde{K}$ is given in \eqref{costantek} with $r=q+1=p$.
\end{thm}
\begin{proof}
If we choose in \eqref{stimabase}
\[
f=u_\infty^{p-1},
\]
where $u_\infty\in W^{1,p}_0(\Omega)$ is a first positive Dirichlet eigenfunction of $\lambda_p^D(\Omega)$, then
\[
u_f=\dfrac{1}{{\lambda^D_p(\Omega)}^{\frac{1}{p-1}}}u_\infty
\]
and we obtain 
\[\dfrac{1}{\lambda_p(\beta,\Omega)^{\frac{1}{p-1}}}\geq \dfrac{1}{{\lambda_p^D(\Omega)}^{\frac{p}{p-1}}}\dfrac{\displaystyle\int_{\Omega}\abs{\nabla u_\infty}^p\;dx}{\displaystyle\int_{\Omega}u_\infty^p\;dx}+\dfrac{\beta^{-\frac{1}{p-1}}}{P(\Omega)^{\frac{1}{p-1}}}\dfrac{\displaystyle\left(\int_{\Omega}u_\infty^{p-1}\:dx\right)^{\frac{p}{p-1}}}{\displaystyle\int_{\Omega}u_\infty^p\:dx}.
\]

By Theorem \ref{reverseholder},
\[
\left(\int_{\Omega}u_\infty^{p-1}dx\right)^{\frac{p}{p-1}}\;dx\geq \dfrac{1}{\Tilde{K}^p}\int_{\Omega}u_\infty^pdx,
\] then
\[
\dfrac{1}{\lambda_p^{\frac{1}{p-1}}(\beta,\Omega)}\ge \dfrac{1}{{\lambda^D}_p(\Omega)^{\frac{1}{p-1}}}+\dfrac{\beta^{-\frac{1}{p-1}}}{P(\Omega)^{\frac{1}{p-1}}}\Tilde{K}^p
\]
that is the thesis.
\end{proof}

For a given open bounded domain $\Omega$, let $u_{\Omega}\in W^{1,p}_0(\Omega)$ be the unique solution to 
\begin{equation*}
\begin{cases}
    -\Delta_pu_{\Omega}=1 & \textrm{in}\; \Omega\\
    u_{\Omega}=0 & \textrm{on}\; \de \Omega.
\end{cases}
\end{equation*} 
The $p-$torsional rigidity is 
\[T_p(\Omega)=\left(\max_{\varphi\in W_0^{1,p}(\Omega)}\dfrac{\left(\ds\int_{\Omega}\abs{\varphi}dx\right)^p}{\displaystyle\int_{\Omega}\abs{\nabla \varphi}^p}dx\right)^{\frac{1}{p-1}}=\int_{\Omega}u_{\Omega}dx=\int_{\Omega}\abs{\nabla u_{\Omega}}^pdx.
\]

\begin{thm}
Let $\Omega \subset \R^N$ be an open bounded set, with $C^{1,\gamma}$ boundary, and let $\beta$ be a positive number. Then,
\begin{equation}
\label{tors1}
\dfrac{1}{\lambda_p(\beta,\Omega)^{\frac{1}{p-1}}}\geq \dfrac{T_p(\Omega)}{\abs{\Omega}}+\left(\dfrac{|\Omega|}{\beta P(\Omega)}\right)^{\frac{1}{p-1}}.
\end{equation}
\end{thm}
\begin{proof}
It is immediate by choosing $f\equiv 1$ in \eqref{stimabase}.
\end{proof}
\begin{rem}
    Using $w=1$ or $w=u_\infty$ as test function in the variational characterization \eqref{variationalaut} of $\lambda_p(\beta,\Omega)$, we immediately get
    \[
    \lambda_p(\beta,\Omega)\le \min\left\{\lambda_p^D(\Omega),\beta\frac{P(\Omega)}{|\Omega|}\right\}.
    \]
    The two inequalities \eqref{upper1} and \eqref{tors1} then give a stronger bound on $\lambda_p(\beta,\Omega)$.
\end{rem}
\begin{rem}
  For $\beta\rightarrow \infty$, by using that the first Robin eigenvalue of the $p$-Laplace operator goes to the first Dirichlet eigenvalue of the $p$-Laplace operator, the inequality \eqref{tors1} becomes
  \[
  \dfrac{T_p(\Omega)\lambda_p^D(\Omega)^{\frac{1}{p-1}}}{|\Omega|} \le 1,
  \]
  which corresponds to a well-known inequality due to P\'olya in the case of the Laplace operator.
\end{rem}

\section{A lower bound for \texorpdfstring{$\lambda_p(\beta,\Omega)$}{TEXT}}
\label{sectionlower}
In this Section, we want to get a lower bound for $\lambda_p(\beta,\Omega)$. 
To this aim, let us define 
\[
   \nu_{p}=
   \inf_{\abs{\nabla v}^{p-2}\frac{\de v}{\de \nu}=K}\dfrac{\displaystyle\int_{\Omega}f^{p'}dx}{\displaystyle\int_{\Omega}\abs{\nabla v}^pdx},
   \]
   where $K=\frac{1}{P(\Omega)}\int_\Omega fdx$, and the infimum is computed among all the functions $f\ge 0$, $f\in L^\infty(\Omega)$ and where $v$ is a solution to the problem \eqref{limitebeta0}.
 \begin{rem}
      It is easily seen that $\nu_p > 0$. Indeed, if we denote with $v_{\de\Omega}=\dashint_{\de  \Omega} v d\mathcal H^{N-1}$, we have $\int_{\de\Omega}(v-v_{\de\Omega})d\mathcal H^{N-1}=0$; by using the weak formulation of \eqref{limitebeta0}, the H\"older inequality and the Poincar\'e-Wirtinger inequality \eqref{pwirt}, we have
      \begin{multline*}
          \int_{\Omega}\abs{\nabla v}^pdx=\int_{\Omega}f(v-v_{\de\Omega})dx\le \left(\int_{\Omega}f^{p'}dx\right)^{\frac{1}{p'}}\left(\int_{\Omega}|v-v_{\de\Omega}|^{p}dx\right)^{\frac{1}{p}}\le \\ \le \Tilde{C}\left(\int_{\Omega}\abs{\nabla v}^pdx\right)^{\frac{1}{p}}\left(\int_{\Omega}f^{p'}dx\right)^{\frac{1}{p'}},
      \end{multline*}
      then
      \[
      \frac{\ds\int_{\Omega}f^{p'}dx}{\ds\int_{\Omega}\abs{\nabla v}^pdx}\ge \frac{1}{\Tilde{C}^{p'}}>0.
      \]
    \end{rem} 
   \begin{thm}
Let $\Omega \subset \R^N$ be an open bounded set, with $C^{1,\gamma}$ boundary. For any $\beta>0$,
 \begin{equation*}
    \dfrac{1}{\lambda_p(\beta,\Omega)}\leq \dfrac{1}{\nu_p}+\dfrac{\abs{\Omega}^{\frac{1}{p-1}}}{\beta^{\frac{1}{p-1}}P(\Omega)^{\frac{1}{p-1}}}.
   \end{equation*}
   \end{thm}
\begin{proof}
Let $f\ge0$, $f\in L^{\infty}(\Omega)$. By using again the convexity of $J_f(\beta)$, we have
\[
J_f(\beta)\leq J_f(\alpha)+\left(\dfrac{1}{\beta^{\frac{1}{p-1}}}-\dfrac{1}{\alpha^{\frac{1}{p-1}}}\right)\int_{\de \Omega}\abs{\abs{\nabla u_{f,\alpha}}^{p-2}\left(\dfrac{\de u_{f,\alpha}}{\de \nu}\right)}^{\frac{p}{p-1}}d\mathcal H^{N-1}.
\]
If we multiply for $\beta^{\frac{1}{p-1}}$,
\[\beta^{\frac{1}{p-1}}J_f(\beta)\leq \beta^{\frac{1}{p-1}}J_f(\alpha)+\left(1-\left(\dfrac{\beta}{\alpha}\right)^{\frac{1}{p-1}}\right)\int_{\de\Omega}\abs{\abs{\nabla u_{f,\alpha}}^{p-2}\left(\dfrac{\de u_{f,\alpha}}{\de \nu}\right)}^{\frac{p}{p-1}}d\mathcal H^{N-1}, 
\]
and using the definition of $J_f(\alpha)$ and the boundary condition 
\[J_f(\beta)\leq \int_{\Omega}\abs{\nabla u_{f,\alpha}}^p\;dx+\dfrac{1}{\beta^{\frac{1}{p-1}}}\int_{\de\Omega}\abs{\abs{\nabla u_{f,\alpha}}^{p-2}\dfrac{\de u_{f,\alpha}}{\de \nu}}^{p'}d\mathcal H^{N-1}.
\]
For $\alpha\rightarrow0^+$, we have that
   \[
    \dfrac{1}{\lambda_p(\beta,\Omega)}\leq \dfrac{1}{\nu_p}+\dfrac{\abs{\Omega}^{\frac{1}{p-1}}}{\beta^{\frac{1}{p-1}}P(\Omega)^{\frac{1}{p-1}}}.
   \]
    Indeed, from  Proposition \ref{limit_prob} we have that 
    \[\lim_{\alpha \rightarrow 0^+}\int_{\Omega}\abs{\nabla u_{f,\alpha}}^p\;dx=\int_{\Omega}\abs{\nabla v}^p\;dx,
    \]
    where $v$ is a solution to \eqref{limitebeta0}, and
    \[\lim_{\alpha \rightarrow 0^+}\int_{\de \Omega}\abs{\abs{\nabla u_{f,\alpha}}^{p-2}\dfrac{\de u_{f,\alpha}}{\de \nu}}^{\frac{p}{p-1}}\;d\mathcal{H}^{N-1}=\lim_{\alpha \rightarrow 0^+}\int_{\de \Omega}\abs{\alpha^{\frac{1}{p-1}} u_{f,\alpha}}^pd\mathcal H^{N-1}=\dfrac{1}{P(\Omega)}\left(\int_{\Omega}f\, dx\right)^{p'}.\]
    Hence, from \eqref{autovalore}
    \[
    \dfrac{1}{\lambda_p(\beta,\Omega)^{\frac{1}{p-1}}}\leq \sup_{\abs{\nabla u_f}^{p-2}\frac{\de u_f}{\de \nu}=K}\dfrac{\ds\int_ \Omega\abs{\nabla v}^pdx}{\displaystyle\int_{\Omega}f^{p'}dx}+\dfrac{1}{\beta^{\frac{1}{p-1}}P(\Omega)^{\frac{1}{p-1}}}\sup_{\int_{\Omega}f^{p'}dx<\infty}\dfrac{\left(\ds\int_{\Omega}f\,dx\right)^{p'}}{\displaystyle\int_{\Omega}f^{p'}dx}.
    \]
    Then
    \[\dfrac{1}{\lambda_p(\beta,\Omega)^{\frac{1}{p-1}}}\leq \dfrac{1}{\nu_p}+\dfrac{1}{\beta^{\frac{1}{p-1}}P(\Omega)^{\frac{1}{p-1}}}\abs{\Omega}^{\frac{1}{p-1}},
    \]
    that is the thesis.
\end{proof}

\section*{Acknowledgement}
This work has been partially supported by PRIN PNRR 2022 ``Linear and Nonlinear PDE’s: New directions and Applications'', and by GNAMPA of INdAM

%
\bibliographystyle{abbrv}

\end{document}